\documentclass{article}
\usepackage[utf8]{inputenc}
\usepackage[english]{babel}
\usepackage{amsmath}
\usepackage{amsfonts}
\usepackage{thmtools}
\usepackage{amssymb}
\usepackage{mathtools}
\usepackage{amsthm}
\usepackage{xcolor}
\usepackage{tikz-cd}

\usepackage[
backend=bibtex,
style=alphabetic,
sorting=ynt
]{biblatex}
\newtheorem{theorem}{Theorem}[section]
\newtheorem{thmx}{Theorem}
\newtheorem{corollary}[theorem]{Corollary}
\newtheorem{lemma}[theorem]{Lemma}

\theoremstyle{definition}

\theoremstyle{definition}

\theoremstyle{definition}
\newtheorem{example}{Example}
\theoremstyle{definition}

\theoremstyle{definition}

\newtheorem{definition}{Definition}
\theoremstyle{proposition}
\newtheorem{proposition}{Proposition}[section]
\theoremstyle{proposition}

\title{Iterated Integrals in Quantitative Topology}
\author{Robin Elliott}

\begin{document}

\maketitle 

\begin{abstract}
Let $X$ be a simply connected Riemannian manifold. Until now, quantitative topology has used Sullivan's rational homotopy theory as the bridge between geometric information on $X$ and torsion-free homotopy theoretic information on $X$. In this paper we introduce Chen's iterated integrals on $\Omega X$ as a new bridge between these two areas. We give two applications: finding upper bounds for Gromov's distortion of higher homotopy groups on $X$, and also proving the non-existence of homologically nontrivial small-volume cycles in the space $\Omega_{\leq L}X$ of loops on $X$ of length at most $L$.
\end{abstract}

\section{Introduction}

This paper provides a new framework to answer questions in quantitative topology using Chen's theory of iterated integrals. In the 1970s Kuo-Tsai Chen built a de Rham complex for the loop space $\Omega X$ of a simply connected smooth manifold. Chen showed that a certain subcomplex of this de Rham complex computes the homology of $\Omega X$, the subcomplex generated by \textit{iterated integrals} on $\Omega X$. Moreover, these iterated integrals have an explicit geometric description, built out of the (usual) de Rham complex on $X$.

We put this framework to use in two independent ways. First we use it to prove the nonexistence of nontrivial cycles in $\Omega X$ that both have small volume and live in a subspace of short loops. Second, we give a new method to upper bound Gromov's distortion in rational homotopy groups of $X$. 

Unless otherwise stated, all homology and cohomology groups are with real coefficients.

\subsection{Loop spaces in quantitative topology}

The role of loop spaces in quantitative topology have been already been considered by Gromov in \cite{metricstructures}. In particular he proves

\begin{theorem} \cite[Theorem~7.3]{metricstructures}
Let $X$ be a compact simply connected Riemannian manifold. There exist constants $C > c > 0$ such that the following holds. For any $L > 0$, let $\Omega_{\leq L}X$ be the subspace of $\Omega X$ of loops of length less than at most $L$. Then the inclusion $\Omega_{\leq L} X \to \Omega X$ induces a surjection on (integer) homology up to degree $cL$. Moreover, the induced map on $H_n$ is zero for $n > CL$.
\end{theorem}

That is, Gromov answers the question: which homology classes in $H_n(\Omega X)$ can be represented by cycles with $\textit{suplength}$ at most $L$, i.e. is contained in the subspace $\Omega_{\leq L}X \subset \Omega X$? Gromov shows that for $n > CL$ the answer is none of them, and for $n < cL$ the answer is all of them.

One could refine this investigation by also keeping track of the \textit{volume} of a chain $Z$ representing a given homology class $\zeta \in \Omega X$, for some notion of volume in $\Omega X$ inherited from the metric $X$.

\begin{example}
Let $n \geq$ and $X=S^n$. In this case, the (integral) homology of $\Omega S^n$ is free rank $1$ in degrees which are a nonnegative multiple of $n-1$, and vanishes otherwise. Let $k \geq 1$ and $\zeta_k \in H_{k(n-1)}(\Omega S^n; \mathbb{R})$ be the image in real homology of a generator in integral homology. For $k=1$ this class can be represented by a sweepout of $S^n$ by loops. For $k=2$ and $n$ even, $\zeta_2$ is a nonzero multiple of the Hurewicz image of the desuspension of an element of $\pi_{2n-1}(S^n)$ with nonzero Hopf invariant. For any $k$ it will be shown that there exists constants $c = c(k)>0$ such that any $k(n-1)$-cycle representing $\zeta_k$ satisfies $\text{Suplength}(Z)^k\text{Vol}(Z) > c$.
\end{example}

In this paper we give bounds of the same flavor for an arbitrary Riemannian manifold $X$ and an arbitrary nonzero $\zeta \in H_n(\Omega X)$. In the case of $X=S^n$, we show these bounds are asymptotically tight. It is not known if these bounds are tight for general $X$. 

\begin{thmx} \label{thma}
Let $\zeta \in H_n(\Omega X)$ be a nonzero homology class. Then there exists a constant $c>0$ and an integer $r > 0$ such that any cycle $Z$ representing $\zeta$ obeys the bound $$ \text{Suplength}(Z)^r \text{Vol}(Z) > c.$$
Moreover, $r$ can be taken to be the minimal nonnegative integer $k$ such that there exists a $\beta \in H^n(\Omega X)$ such that $\langle \beta,\; \zeta \rangle \neq 0$ and $\beta$ can be represented closed iterated integral with each summand built out of at most $k$ differential forms on $X$.
\end{thmx}

Informally, if we are looking for representatives of $\zeta$ in $\Omega_{\leq L}$ then the volume of such representatives are bounded below. Conversely, if $Z$ represents $\zeta \in H_*(\Omega X)$ with very small volume then $Z$ must have very large suplength.

\subsection{Iterated integrals and distortion in $\pi_n(X) \otimes \mathbb{Q}$}

Our second application of the framework is a new method to provide upper bounds to Gromov's distortion of homotopy groups. Following \cite{scalspaces}, define the \textit{distortion} of $\alpha \in \pi_n(X) \otimes \mathbb{Q}$ to be 
$$ \delta_{\alpha}(L) \coloneqq \text{sup}\{k \;|\;\text{there exists an } L \text{-Lipschitz map } S^n \to X \text{ with } [f] = k\alpha \}.$$

Gromov outlines, and Manin fleshes out, an algorithm using Sullivan's minimal models and obstruction theory to find upper bounds for $\delta_{\alpha}(L)$. Here we present a new method for obtaining bounds.

\begin{thmx} \label{thmb}
Suppose $X$ is a simply connected Riemannian manifold and $\alpha \in \pi_n(X) \otimes \mathbb{Q}$. Denote by $\tau(\alpha)$ the image of $\alpha$ under the composite of $$\pi_n(X) \otimes \mathbb{Q} \xrightarrow{\sim} \pi_{n-1}(\Omega X) \otimes \mathbb{Q} \xrightarrow{Hurewicz} H_{n-1}(\Omega X; \mathbb{Q}).$$ Suppose $\beta$ is an iterated integral on $\Omega X$ with each summand built out of at most $r$ differential forms on $X$, and also such that $\langle \beta, \tau(\alpha) \rangle \neq 0$. Then the distortion of $\alpha$ is $O(L^{n-1+r})$.
\end{thmx}

\begin{example} Consider the space $X = S_a^3 \vee S_b^3 \cup_{[a,[a,b]]} D^8 \cup_{[b,[a,b]]} D^8$ considered in \cite[\S 13(d),(e)]{fht} and \cite[Section~5]{shadows}. The rational homotopy group $\pi_{10}(X) \otimes \mathbb{Q}$ is rank one; let $\tau$ be a generator. Manin follows the method of obstruction theory to give an upper bound on the distortion of $\tau$ of $O(L^{12})$, and then refines this to an upper bound of $O(L^{11})$. In Section 5, we give another argument that the distortion of $\tau$ is $O(L^{11})$, avoiding the obstruction theory.
\end{example}

In broad terms, the obstruction theory of \cite{shadows} is replaced with the method of \textit{weight reduction}, as explained in \cite{kozsulduality}, \cite{sinhahopfinvs}. Section 5 contains examples of such computations. 

Theorem 1.2 and Theorem 1.3 both follow from the main theorem, which is proved in Section 4.

\begin{thmx}\label{thmc}
Let $X$ be a simply connected compact Riemannian manifold and $\beta \in H^n(\Omega X; \mathbb{R})$. Then there exists a differential form $\omega$ on $\Omega X$ representing $\beta$ which admits a positive integer $r$ and a $C > 0$ such that for all $\gamma \in \Omega X$,
$$ ||\omega(\gamma)||_{\infty} \leq C\text{Length}(\gamma)^r. $$
\end{thmx}

\subsection{Outline of the paper}

The outline of the paper is as follows. In Section 2 we give an expository introduction to iterated integrals in the setting of the loop space of $S^n$. In Section 3 we present the necessary background; the precise definitions of iterated integrals on loop spaces and the metric data that will be kept track of. Section 4 contains the proof of Theorem \ref{thmc} and from this deduces Theorems \ref{thmb} and \ref{thma}. In Section 5 we present examples and computations using the bar complex of a minimal model. In Section 6 the sharpness of the bounds given by Theorem \ref{thmc} are discussed.

\subsection{Acknowledgements}
I would like to thank my advisor, Larry Guth, for tirelessly helping me as my advisor. I would like to also thank Fedya Manin for support, conversations and guidance, as well as giving me opportunities to talk about this work. I would like to thank Dev Sinha for expositing similar ideas in the Poincare dual setting of submanifolds instead of forms. Finally, I would like to thank Luis Kumandari and Sasha Berdnikov for helpful conversations.

\section{Warmup: homotopy functionals on $S^n$}

In this section we give an expository introduction to differential forms on the loop space, saving the groundwork of setting up the machinery to the next section.

\subsection{The degree of maps $S^n \to S^n$}

The simplest example of a homotopy functional is the degree of a map $S^n \to S^n$. This is detected by a volume form $\omega$ on $S^n$: given $f:S^n \to S^n$, $\text{deg}(f) = \int_{S^n}f^*\omega$. In fact pulling back the volume form detects volume \textit{locally}: for any $f: P \to S^n$ with $P$ an $n$-dimensional manifold, $\text{Vol}(f) = \int_P f^*\omega$. 

In particular, if $P = [0,1] \times Q$, with $f: [0,1] \times Q \to S^n$ sending $\{0,1\} \times Q$ to the basepoint $x_0 \in S^n$, then $f$ desuspends (under the suspension-loop adjunction) $\hat{f}: Q \to \Omega S^n$ in the following way. Fix once and for all smooth sweepouts of spheres $S^n$ by loops, i.e. unit maps $\eta: S^{n-1} \to \Omega S^n$ of the suspension-loop adjunction. Then we can take $\hat{f} \coloneqq (\Omega f) \circ \eta$. We will define an $(n-1)$-form $\smallint \omega$ on $\Omega S^n$ such that pulling back $\smallint \omega$ by $\hat{f}$ detects the volume of $f$.

Let $\gamma \in \Omega^n$ and $V_1, \dots, V_{n-1}$ be vector fields along $\gamma$ (which should be thought of as tangent vectors in $T_{\gamma}\Omega S^n$. Then define
the $(n-1)$-form $\smallint \omega$ by
$$(\smallint \omega)(\gamma)(V_1, \dots, V_{n-1}) \coloneqq \int_{t=0}^{t=1} \omega(\gamma(t))(\gamma'(t), V_1(t), \dots, V_{n-1}(t))$$

Then $\int_{Q} \hat{f}^*(\smallint \omega)) = \int_{Q\times I} f^* \omega =  \int_P f^*\omega  = \text{Vol}(f)$. This proves the following proposition.

\begin{proposition}
Let $\omega$ be a volume form on $S^n$. Then the $(n-1)$-form $\int \omega$ on $\Omega S^n$ detects the degree of a smooth map $f: S^n \rightarrow S^n$. That is, if $\hat{f}: S^{n-1} \rightarrow \Omega S^n$ is the desuspension of $f$ under the suspension-loop adjunction then
\begin{align*} \text{deg}(f) = \langle \hat{f}^*(\smallint \omega), \; [S^{n-1}] \rangle.
\end{align*}
\end{proposition}

Before proceeding, we give an equivalent definition of $\smallint \omega$. This definition will be better suited to generalise to other more complicated homotopy functionals. Let $I \coloneqq [0,1]$ and $\text{ev}: \Omega S^n \times I \to S^n$ be the map $(\gamma, t) \mapsto \gamma(t)$. The volume form $\omega$ on $S^n$ pulls back to an $n$-form $\text{ev}^*\omega$ on $\Omega S^n \times I$. Integrating this over the fiber of the projection $\Omega S^n \times I \to \Omega S^n$ gives an $(n-1)$-form which is equal to $\smallint \omega$.

\subsection{The Hopf invariant of maps $S^{2n-1} \to S^n$}

The next simplest example of a homotopy functional is the Hopf invariant $\pi_{2n-1}(S^n) \to \mathbb{Z}$ for $n$ even. We will define a differential form of degree $(2n-2)$ on $\Omega S^n$ detecting this. 

Let $\Delta^2 \subset I \times I$ be the set of $(t_1, t_2)$ such that $t_1 \leq t_2$. There is a map $\text{ev}_2: \Omega S^n \times \Delta^2 \to S^n \times S^n$ sending $(\gamma, (t_1, t_2))$ to $(\gamma(t_1), \gamma(t_2))$. The volume form $\omega \times \omega$ on $S^n \times S^n$ pulls back to a $(2n)$-form $\text{ev}^*(\omega \times \omega)$ on $\Omega S^n \times \Delta^2$. Integrating this over the fiber of the projection $\Omega S^n \times \Delta^2 \to \Omega S^n$ gives a $(2n-2)$-form on $\Omega S^n$ which we will denote $\smallint \omega \omega$.

\begin{proposition}
Let $\omega$ be a volume form on $S^n$ with $n$ even. Then the $(2n-2)$-form $\int \omega \omega$ on $\Omega S^n$ detects the Hopf invariant of a smooth map $f: S^{2n-1} \rightarrow S^n$. That is, if $\hat{f}: S^{2n-2} \rightarrow \Omega S^n$ is the adjoint of $f$ under the suspension-loop adjunction then
\[ \text{Hopf}(f) = \langle \; \hat{f}^*(\smallint \omega \omega), \; [S^{2n-2}] \; \rangle \]
\end{proposition} 

\begin{proof} Consider the commutative diagram

\[\begin{tikzcd}
S^{2n-2} \times \Delta^2 \arrow[r, "\hat{f}\times \text{id}"] \arrow[d, "\pi"]
& \Omega S^n \times \Delta^2 \arrow[d, "\pi"] \arrow[r, "\text{ev}_2"] &  S^n \times S^n \\
S^{2n-2} \arrow[r, "\hat{f}"]
& \Omega S^n
\end{tikzcd}\]

Here $\pi$ is the projection onto the first factor inducing a covariant integration over the fiber map $\pi_*$ sending $k$-forms to $(k-2)$-forms. By definition the iterated integral $\smallint \omega \omega = \pi_* \text{ev}_2^* (\omega \times \omega)$ and by commutativity of the diagram we have
\[ \hat{f}^*(\smallint \omega \omega) = \pi_* (\hat{f} \times \text{id})^* \text{ev}_2^* (\omega \times \omega). \]

For ease of notation, denote $\text{ev}_2 \circ (\hat{f} \times \text{id}): S^{2n-2} \times \Delta^2 \to S^n \times S^n$ by $g$. This is the map $(x, s, t) \mapsto (f'(x, s), f'(x, t))$ where $f': S^{2n-2} \times I \to S^{2n-1}$ is $f$ composed with the quotient map $S^{2n-2} \times I \to S^{2n-1}$.

For any $2n$-form $\alpha$ on $S^{2n-2}\times \Delta^2$,
\[ \langle [S^{2n-2}], \pi_*\alpha \rangle = \int_{S^{2n-2}} \int_{\Delta^2} \alpha = \int_{S^{2n-2} \times \Delta^2} \alpha \]

So it remains to show that $\text{Hopf}(f) = \int_{S^{2n-2}\times\Delta^2} g^*(\omega \times \omega)$. We will argue this in the Poincar\'e dual setting, i.e. consider the preimage under $g$ of a regular point $(p, q) \in S^n \times S^n$ and show that its signed count is equal to the linking number definition of the Hopf invariant of $f$.

Consider a chain $Z$ representing the fundamental class of $S^{2n-2} \times {\Delta^2}$, with $\partial Z = S^{2n-2} \times \partial \Delta^2$. Then since $\Delta^2 \coloneqq \{0 \leq s \leq t \leq 1\}$, the boundary $\partial \Delta^2$ splits into the sum $B + D$, where $B = \{s = 0\} \cup \{t = 1\}$ and $D = \{s = t\}$. Both $g_*(S^{2n-2} \times B)$ and $g_*(S^{2n-2} \times D)$ are $(2n-1)$-cycles lying in $n$-dimensional subspaces of $S^n \times S^n$; they lie in $(S^n, *) \cup (*, S^n)$ and $\text{Diag}(S^n)$ respectively. So they can both be bounded by $2n$ chains of zero volume, call them $Z_B$ and $Z_D$ respectively. Then $Z + Z_B + Z_D$ is a $2n$-cycle in $S^n \times S^n$. We will now show that it is homologus to $\text{Hopf}(f)$ times the fundamental class of $S^n \times S^n$.

Pick a regular point $(p, q) \in S^n \times S^n$ with neither $p$ nor $q$ the basepoint. Then $Q \coloneqq (g|_{S^{2n-2}\times D})^{-1}(S^n \times \{q\})$ and $P \coloneqq (g|_{S^{2n-2}\times D})^{-1}(\{p\} \times S^n)$ are two $(n-1)$-manifolds in $S^{2n-2} \times D \cong S^{2n-2} \times (0, 1)$. We will show that the linking number of $P$ and $Q$ is equal to the signed count of $g^{-1}(p, q)$, which in turn is equal to the degree of our $(2n)$-cycle $Z + Z_B + Z_D$.

Consider $\tilde{P} \coloneqq g^{-1}(\{p\} \times S^n)$ and $\tilde{Q} \coloneqq g^{-1}(S^n \times \{q\})$, submanifolds of $S^{2n-2} \times \Delta^2$. Note that $\tilde{P} \cap \tilde{Q} = g^{-1}(p,q)$. Under the projection $\pi: \Delta^2 \to D$ given by $(s, t) \mapsto (t, t)$, $\tilde{P}$ maps to a submanifold $\partial^{-1}P$ of $S^{2n-2}\times D$ bounding $P$, and $\tilde{Q}$ maps to $Q$. So the linking number of $P$ and $Q$ (the intersection number of $\partial^{-1}P$ and $Q$) is equal to the intersection number of $\tilde{P}$ and $\tilde{Q}$ which is equal to the signed count count of $ g^{-1}(p,q)$ as required.
\end{proof}

\section{Setup}

\subsection{Iterated integrals}

In this we will give the necessary background for and definition of iterated integrals, following Hain's thesis \cite{hainthesis} and \cite{itintreview}.

\begin{definition} A \textit{differential space} is a Hausdorff space $X$ together with a family $\{\phi_i: U_i \rightarrow X\}_{i \in I}$ of continuous maps, called \textit{plots}, subject to the following conditions. First, the $U$ are convex subsets of Euclidean space (of any dimension $n \geq 0$) with nonempty interior. Secondly, every map $\{*\} \rightarrow X$ is a plot. Thirdly, if $f: U \rightarrow U'$ is smooth in the classical sense and $\phi: U' \rightarrow X$ is a plot, then $\phi f: U \rightarrow X$ is also a plot.
\end{definition}

Hence the plots should be thought of as determining which functions into $X$ are smooth. More precisely, a function $f: X \rightarrow Y$ is said to be \textit{smooth} if and only if $f$ pushes forward all plots on $X$ to plots on $Y$. Note that the notion of plots is more general than the notion of charts on (finite dimensional) manifolds, since they not required to be local homeomorphisms. 

A differential form on a differential space $X$ is then specified by its pullbacks on all plots. Formally: \begin{definition}A \textit{differential $k$-form} $\omega$ on a differential space $X$ is the assignment of a $k$-form $\omega_{\alpha}$ on $U$ for each plot $\alpha: U \rightarrow X$ which are compatible in the following sense. If $\phi: U \rightarrow U'$ is smooth and $\alpha: U' \rightarrow X$ is a plot, then $\phi^*\omega_{\alpha} = \omega_{\alpha\phi}$. One can think of the $\omega_{\alpha}$ as the pullbacks of the form $\omega$ along the plot $\alpha$. The usual definitions of addition, wedge product and exterior derivative on forms allow us to define the \textit{de Rham complex} of a differential space $X$; it is a graded differential algebra denoted $\Lambda^*X$.
\end{definition}

Note that a smooth manifold $X$ is a differential space, where a map $\alpha: U \rightarrow X$ is a plot if and only if it is smooth. Then the de Rham complex of $X$ in the sense of Chen is isomorphic to the classical de Rham complex.

If $X$ is a differential space with basepoint $x_0$, then let $\Omega X$ denote the set of functions $\gamma: [0,1] \rightarrow X$ which are piecewise plots into $X$ and such that $\gamma(0) = \gamma(1) = x_0$. Then $\Omega X$ has the structure of a differential space generated by the following maps: $\phi: U \to \Omega X$ which admit a partition $0 = t_0 < t_1 < \dots < t_m = 1$ of $[0,1]$ such that $\hat{\phi}|_{U \times [t_{j-1}, t_j]} :U \times [t_{j-1}, t_j] \to X$ is a plot on $X$ for each $j$.

Furthermore, subspaces of differential spaces are also differential spaces, as are the product of two differential spaces is again a differential space \cite[Chapter~4.4~(c),(d)]{hainthesis}.

\begin{definition}
Let $\Delta^n$ denote the $n$-simplex
$$ \Delta^n \coloneqq \{(t_1, t_2, \dots, t_n) \subset \mathbb{R}^n \; | \; 0 \leq t_1 \leq \dots \leq t_n \leq 1 \}.$$
Then the \textit{evaluation map} is a smooth map between differential spaces:
\begin{align*}
\text{ev}_n: \Omega X \times \Delta^n & \to X^{\times n} \\ (\gamma, (t_1, \dots, t_n)) & \mapsto (\gamma(t_1), \dots, \gamma(t_n))
\end{align*}
\end{definition}

\begin{definition}
Let $X$ be a differential space. Just as for forms on finite dimensional manifolds, we can define an \textit{integration over the fibers} map for differential spaces:
$$\int_{\Delta^n}: \Lambda^k(X \times \Delta^n) \to \Lambda^{k-n}(X)$$
The definition is plotwise. Let $\omega$ be a differential form on $X \times \Delta^n$. Given a plot, $\alpha: U \to X$ on $X$, define $(\int_{\Delta_n} \omega)_{\alpha}$ to be the form on $U$ 
$$ \int_{\Delta_n} (\omega_{\alpha \times \text{id}_{\Delta^n}})  $$
Here it is used that $\alpha \times \text{id}_{\Delta^n}$ is a plot on $X \times \Delta^n$. The integration is over the volume element $dt_1 \wedge \dots dt_n$ on $\Delta^n$.
\end{definition}

The preceding two definitions allow us to now give the definition of iterated integrals.

\begin{definition} (Iterated integrals) Let $r$ be a nonnegative integer. Let $\omega_1, \dots \omega_r$ be forms of dimension $n_1, \dots, n_r$ on the manifold $M$. Then the \textit{iterated integral} $\smallint \omega_1 \dots \omega_r$ is a $(-r+\Sigma n_i)$-form on $\Omega M$ given by $$\smallint \omega_1 \dots \omega_r \coloneqq \int_{\Delta^r} \text{ev}_r^*(\omega_1 \times \dots \times \omega_r).$$ Here $\omega_1 \times \dots \times \omega_r$ is the product form $\pi_1^*\omega_1 \wedge \dots \wedge \pi_r^*\omega_r$ on $M^{\times r}$. The iterated integral is said to be of \textit{length} $r$.
\end{definition}

\subsection{Properties of iterated integrals}

Let $X$ be a simply connected compact manifold. Let $\smallint\Lambda^*(X)$ denote the differential graded subalgebra of $\Lambda^*(\Omega X)$ generated by iterated integrals. 

Let $C_*(\Omega X)$ denote the chain complex of smooth chains on $\Omega X$. There is an integration map\footnote{ also referred to as the \textit{Stokes map} in \cite{itintreview}} 
\begin{align*}
\rho: \Lambda^n(\Omega X) \to \;& \text{Hom}(C_n(\Omega X); \mathbb{R}) \\
\omega \mapsto \;& \{(\sigma: \Delta^n \to \Omega X) \mapsto \int_{\Delta^n}\sigma^*\omega\}
\end{align*}

\begin{theorem}(Chen's loop space de Rham theorem \cite{itpathints}) Let $X$ be as above. The integration map restricted to $\smallint\Lambda^*(X)$ induces an isomorphism
$$\rho_*: H^*(\smallint\Lambda^*(X)) \xrightarrow{\sim} H^*(\Omega X; \mathbb{R}).$$
\end{theorem}

So any cohomology class of $\Omega X$ (and similarly any homotopy functional on $X$) can be represented by an iterated integral.

In fact, Chen gives a stronger version of this:
\begin{theorem} \cite[Thm~2.3.1]{itpathints} Let $X$ be as above. Let $A$ be a subalgebra of $\Lambda^*(X)$ such that the integration map $A \to C^*(X; \mathbb{R})$ induces an isomorphism on cohomology. Let $\smallint A$ denote the differential graded subalgebra of $\smallint \Lambda^*(X)$ generated by iterated integrals of forms in $A$. Then the integration map restricted to $\smallint A$, gives an isomorphism
$$\rho_*: H^*(\smallint A) \xrightarrow{\sim} H^*(\Omega X; \mathbb{R}).$$
\end{theorem}

An example of such an $A$ that will be used in Section 5 is the image of a minimal model $m_X: \mathcal{M}_X \to \Lambda^*(X)$.

The final property of iterated integrals that we will use is their relation to the bar construction.

Recall the definition of the bar construction of a differential graded algebra $A$ as follows. It is a differential graded algebra $B(A)$ that, as an algebra, is given by $$ B(A) \coloneqq \bigotimes_{r\geq 0} A^{>0}$$
with an element $a_1 \otimes \dots \otimes a_r$ denoted by $a_1 | \dots | a_r$ and has degree $(\text{deg}(a_1) + \dots + \text{deg}(a_r) - r$.
The differential on $B(A)$ is 
$$ d(|a_1| \dots | a_r|) = - \sum_{i=1}^{r} (-1)^{n_i}|a_1| \dots | da_i | \dots | a_r| + \sum_{i=2}^{r} (-1)^{n_i} |a_1| \dots |a_{i-1} a_i | \dots | a_r | $$
where $n_i = \sum{j<i}(\text{deg}(a_j)-1)$.

\begin{theorem} [Iterated Path Integrals, Thm 4.1.1]
Let $X$ be a connected differential space with $H_0(X) = \mathbb{Z}$. Let $A$ be a differential graded subalgebra of $\Lambda^*(X)$ with $A^0 = \mathbb{R}$ and $A^1 \cap d\Lambda^0(X) = 0$. Then the map $B(A) \to \smallint A \subset \Lambda^*(\Omega X)$ given by $$ | \omega_1 | \dots | \omega_r | \mapsto \smallint \omega_1 \dots \omega_r$$ is an isomorphism of differential graded algebras.
\end{theorem}

\subsection{Metric setup on $\Omega X$}

Let $(X,g)$ be a Riemannian manifold. The Riemannian distance function $d_X$ on $X$ allows us to define a distance function on $\Omega X$ as follows. Given loops $\gamma_1, \gamma_2$, define
$$d(\gamma_1, \gamma_2) = \sup_{t\in[0,1]} d_X(\gamma_1(t), \gamma_2(t)).$$
This gives a metric on $\Omega X$, which which allow us to define the notion of volume for chains on $\Omega X$.

By the \textit{volume} (or \textit{$n$-volume}) of a simplex $\sigma: \Delta^n \to \Omega X$, denoted $\text{Vol}(\sigma)$, we will mean the $n$-dimensional Hausdorff measure of the image of $\sigma$ with respect to the metric on $\Omega X$. Analogously, an arbitrary chain $\sum q_i \sigma_i$ (with $q_i \in \mathbb{R}$) has volume $\sum |q_i| \text{Vol}(\sigma_i)$.

There is also a function $\text{Length}: \Omega X \to \mathbb{R}$ which associates to a loop $\gamma$ its length when viewed as a path in $X$.  Then the \textit{suplength} of a simplex $\sigma$ is the supremum of the lengths of all of the loops in the image of $\sigma$. An arbitrary chain $\sum q_i \sigma_i$ has suplength equal to the maximum suplength of the $\sigma_i$ that have nonzero coefficient $q_i$.

Given a differential space $X$ with metric, such as a Riemannian manifold, or a loop space of either of such with the induced metric, we will define a norm on forms analogous to the norm on forms on a Riemannian manifold.

Recall the data of a $k$-form $\omega$ on $X$ is a compatible collection of $k$-forms $\omega_{\phi_i}$ on $U_i$ for each plot $\phi_i: U_i \to K$. Endow each $U_i$, a convex subset of Euclidean space, with the Euclidean metric. 

At a given point $x \in X$, define $||\omega(x)||_{\infty}$ to be the quantity
$$ ||\omega(x)||_{\infty} \coloneqq \sup \{ \dfrac{||\omega_{\phi}(y)||_{\infty}}{\text{Dil}_k(\phi)} \;|\; \phi: U \to X \text{ a plot with Dil$_k(\phi) \neq 0$, and $\phi(y) = x$}\} $$

Here $\text{Dil}_k(\phi)$ is the \textit{$k$-dilation} of $\phi$: say $\phi$ has $k$-dilation $\leq \lambda$ if each $k$-dimensional surface $\Sigma$ in $U$ has image in $X$ with Hausdorff $n$-measure at most $\lambda \text{Vol}_k(\Sigma)$.  Note this is consistent with $||\cdot||_{\infty}$ defined on Riemannian manifolds, since in this setting $\omega_{\phi} = \phi^*(\omega)$ and for $k$-forms $\omega$, $||\phi^*(\omega)||_{\infty} \leq \text{Dil}_k(\phi)||\omega||_{\infty}$.

There are two key properties of this norm that we will use, both generalizations to the properties of the norm on forms on Riemannian manifolds. 

\begin{lemma} If $f: X \to Y$ is a smooth $L$-Lipschitz map between metric differential spaces, then for any $n$-form $\alpha$ on $Y$, $||f^*\alpha(x)||_{\infty} \leq L^n||\alpha(f(x))||_{\infty}$.
\end{lemma}

\begin{proof}
The ingredients here are that if $\phi: U \to X$ is a plot then so is $f \circ \phi$, and also $\text{Dil}_k(f \circ \phi) \leq L^n \text{Dil}_k(\phi)$. Then
\begin{align*}
    ||f^*\alpha(x)||_{\infty} &= \sup \{ \frac{||(f^*\alpha)_{\phi}(y)||_{\infty}}{\text{Dil}_k{\phi}} \;\;|\;\; \phi: U \to X,\; \phi(y) = x,\;\; \text{Dil}_k(\phi) \neq 0 \} \\
    &= \sup \{ \frac{||\alpha_{f \circ \phi}(y)||_{\infty}}{\text{Dil}_k{\phi}} \;\;|\;\; \phi: U \to X,\; \phi(y) = x,\;\; \text{Dil}_k(\phi) \neq 0 \} \\
    &\leq L^n \sup \{ \frac{||\alpha_{f \circ \phi}(y)||_{\infty}}{\text{Dil}_k{\phi}} |\; f \circ \phi: U \to Y,\; f\circ \phi(y) = f(x),\; \text{Dil}_k(f \circ \phi) \neq 0 \} \\
    & \leq L^n ||\alpha(f(x))||_{\infty}.
\end{align*}

\end{proof}

\begin{lemma} Let $X$ be a metric differential space. Let $Z$ be a $k$-chain in $X$ and $\omega$ a $k$-form on $X$. Then the integration pairing satisfies
$$| \langle \omega, Z \rangle| \leq \text{Vol}(Z) \cdot \sup_{x \in \text{supp}(Z)}||\omega(x)||_{\infty}.$$
\end{lemma}

\begin{proof} Since $Z$ is finite dimensional, upon pulling back to $Z$ this immediately reduces to the finite dimensional case.
\end{proof}

\section{Proof of Theorems \ref{thmc}, \ref{thmb}, and \ref{thma}}

In this section we prove Theorem \ref{thmc} and use this to deduce Theorems \ref{thmb} and \ref{thma}.

\begin{theorem} (Theorem \ref{thmc})
Let $(X,g)$ be a compact simply connected Riemmanian manifold, and $\omega_1, \dots, \omega_r$ forms on $X$. Then for any $\gamma \in \Omega X$ the form $\smallint \omega_1 \dots \omega_r$ on $\Omega X$ satisfies
$$ ||(\smallint \omega_1 \dots \omega_r)(\gamma)||_{\infty} \leq \frac{1}{r!} \text{Length}(\gamma)^r ||\omega_1||_{\infty} \dots ||\omega_r||_{\infty} $$
\end{theorem}

\begin{proof}
Let $k$ be the dimension of the form $\smallint \omega_1 \dots \omega_r$. Let $\phi: U \to \Omega X$ be a plot with nonzero $k$-dilation and image containing $\gamma$. Let $g$ denote the composite 
$$ U \times \Delta^r \xrightarrow{\phi \times \text{id}_{\Delta^r}} \Omega X \times \Delta^r \xrightarrow{\text{ev}_r} X^{\times r}$$

It suffices to show that at a point $x \in \phi^{-1}(\gamma) \subset U$,
$$ || \int_{\Delta^r} g^*(\omega_1 \times \dots \times \omega_r)(x) ||_{\infty} \leq \text{Dil}_k(\phi) \cdot \frac{1}{r!} \text{Length}(\gamma)^r \cdot ||\omega_1 \times \dots \times \omega_r ||_{\infty}$$

Let $v_1, \dots, v_k$ be vectors in $T_x U$ with unit norm (under the Euclidean metric on U).

Then
\begin{align*}
    & |\int_{\Delta^r} g^*(\omega_1 \times \dots \omega_r)(x)[v_1, \dots, v_k]| \\
    &= |\int_{(t_i) \in \Delta^r} g^*(\omega_1 \times \dots \times \omega_r)(x, (t_i))[\tilde{v}_1, \dots, \tilde{v}_k, \partial \tilde{t}_1, \dots, \partial \tilde{t}_r]| \\
    \intertext{Where $\partial t_i$ are the standard basis vectors of $T_{(t_i)}\Delta^r$, $\tilde{v}_i$ is a lift of $v_i$ to $T_{(x,t)}(U \times \Delta^r)$ and similarly for $\partial \tilde{t}_i$.}
    &\leq \int_{(t_i) \in \Delta^r}||\omega_1 \times \dots \times \omega_r|| \cdot \text{Dil}_k(\phi) \cdot |\dot{\gamma}(t_1)| \cdot \ldots \cdot |\dot{\gamma}(t_r)| \\
    \intertext{Since $g: U \times \Delta^r \to X^{\times r}$ has $k$-dilation $\text{Dil}_{\phi}$ in the first factor, and $g_*(\partial \tilde{t}_j)$ has norm $\dot{\gamma}(t_j)$,}
    &= ||\omega_1 \times \dots \times \omega_r|| \cdot \text{Dil}_k(\phi) \cdot \int_{(t_i) \in \Delta^r} |\dot{\gamma}(t_1)| \cdot \ldots \cdot |\dot{\gamma}(t_r)| \\
    &= ||\omega_1 \times \dots \times \omega_r|| \cdot \text{Dil}_k(\phi) \cdot \frac{1}{r!}\text{Length}(\gamma)^r
\end{align*}
as required.

\end{proof}

Theorems \ref{thmb} and \ref{thma} follows directly from this and Theorem 3.1.

\begin{corollary}
(Theorem \ref{thmb}) Given $\alpha \in \pi_n(X) \otimes \mathbb{Q}$, the distortion of $\alpha$ is $O(L^{n-1+r})$ where $r$ is the smallest positive integer such that $\alpha$ is detected by a cohomology class in $H^{n-1}(\Omega X; \mathbb{Q})$ which can be represented by a sum of iterated integrals all of whose summands are length at most $r$.
\end{corollary}

\begin{proof} \textit{(of Corollary)}
Let $f: S^n \to X$ be an $L$-Lipschitz map such that $[f] = k\alpha \in \pi_*(X) \otimes \mathbb{Q}$. 

Let $\tau$ be the composite $\pi_n(X) \otimes \mathbb{Q} \xrightarrow{\sim} \pi_{n-1}(\Omega X) \otimes \mathbb{Q} \to H_{n-1}(\Omega X; \mathbb{Q})$. This map $\tau$ can also be refined to a map $\tau: \text{Map}(S^n, X) \to C_{n-1}(\Omega X)$, defined as follows. Fix a cycle $\zeta_n$ generating $H_{n-1}(\Omega S^n)$. Then define $\tau(f) \coloneqq (\Omega f)_*\zeta_n$. Let $\omega$ be a differential $(n-1)$-form on $\Omega X$ such that $\langle \omega, \tau(\alpha) \rangle = 1$. By Theorem 3.1 we can assume $\omega$ is a sum of iterated integrals. Then by the hypothesis of the corollary we can further assume that there is a constant $C > 0$ such that for all $\gamma \in \Omega X$, $$||\omega(\gamma)||_{\infty} \leq C \text{Length}(\gamma)^r.$$

Now, suppose $\zeta_n$ has suplength $L_0$ and volume $V_0$. Then $(\Omega f)_*(\zeta_n)$ is an $(n-1)$-cycle in $\Omega X$ of volume at most $L^{n-1}V_0$ and suplength $LL_0$.

Then
\begin{align*}
k = \langle \omega, (\Omega f)_*(\zeta_n) \rangle &\leq \text{Vol}((\Omega f)_*(\zeta_n) \cdot \sup_{\gamma \in \text{supp}(\tau(f))} ||\omega(\gamma)||_{\infty} \\ & \leq L^{n-1}V_0 \cdot C(LL_0)^r \\ &= O(L^{n-1+r})
\end{align*}
as required.
\end{proof}

\begin{corollary} (Theorem \ref{thma}) Let $\zeta \in H_n(\Omega X)$ be a nonzero homology class. Then there exists a constant $c>0$ and an integer $r > 0$ such that any cycle $Z$ representing $\zeta$ obeys the bound $$ \text{Suplength}(Z)^r \text{Vol}(Z) > c.$$
Moreover, $r$ can be taken to be the minimal nonnegative integer $k$ such that there exists a $\beta \in H^n(\Omega X)$ such that $\langle \beta,\; \zeta \rangle \neq 0$ and $\beta$ can be represented closed iterated integral with each summand built out of at most $k$ differential forms on $X$.
\end{corollary}

\begin{proof} Let $\beta$ be some class in $H^n(\Omega X)$ such that $\langle \beta, \zeta \rangle = 1$. This $\beta$ can be represented by a differential form $\omega$ which is a sum of iterated integrals, so by Theorem \ref{thmc} there exists a constant $C > 0$ such that for any $\gamma \in \Omega X$, $$||\omega(\gamma)||_{\infty} \leq C \text{Length}(\gamma)^r$$ where $r$ is the length of the longest iterated integral summand in $\omega$.

Then $$1 = \langle \omega, Z \rangle \leq \text{Vol}(Z) \sup_{\gamma \in \text{supp}(Z)} ||\omega(\gamma)||_{\infty} \leq \text{Vol}(Z) \cdot C \; \text{Suplength}(Z)^r$$ so $\text{Suplength}(Z)^r \text{Vol}(Z) \geq 1/C$ as required.
\end{proof}

\section{Applications to Gromov's distortion of homotopy groups}

In this section we compute some examples of iterated integrals on simply connected Riemannian manifolds $K$.

\begin{example}
Let $K = S^n$ be a sphere and $\omega$ a volume form on $S^n$. In section 2, we demonstrated that iterated integral $\smallint \omega$ that detected the degree functional $\pi_n(S^n) \otimes \mathbb{Q} \to \mathbb{Q}$ and for even $n$ the iterated integral $\smallint \omega \omega$ detects the Hopf invariant $\pi_{2n-1}(S^{2n}) \otimes\mathbb{Q} \to \mathbb{Q}$. By the theorem, these differential forms on $\Omega S^n$ satisfy, for $\gamma \in \Omega S^n$:
\begin{align*}
    ||\smallint \omega(\gamma)||_{\infty} \; & \lesssim \text{Length}(\gamma) \\
     ||\smallint \omega \omega(\gamma)||_{\infty} \; & \lesssim \text{Length}(\gamma)^2
\end{align*}
Hence by the corollary we recover Gromov's original bounds for these homotopy functionals: if $f: S^n \to S^n$ is $L$-Lipschitz, then $|\text{deg}(f)| = O(L^n)$. If $n$ is even and $f: S^{2n-1} \to S^n$ is $L$-Lipschitz, then $|\text{Hopf}(f)| = O(L^{2n})$.
\end{example}

\begin{example}
Let $K = \mathbb{CP}^n$. Let $\omega$ be a $2$-form generating $H^2(K)$. The iterated integral $\smallint \omega \omega^n$ of length $2$ is a closed $(2n)$-form generating $H^{2n}(\Omega \mathbb{CP}^n)$ \cite{hainthesis} and hence detects the generator of $\pi_{2n+1}(\mathbb{CP}^n)$.  Thus the distortion of this generator is $O(L^{2n+2})$.
\end{example}

Next we give the example of the $8$-dimensional nonformal cell complex $X$ considered in \cite[\S5]{shadows}. An alternate approach to obstruction theory is given to show that the the generator of $\pi_{10}(X) \otimes \mathbb{Q}$ has distortion $O(L^{11})$.

\begin{example}
Let $K$ be a Riemannian manifold\footnote{Created, for example, by thickening the cell complex} equivalent to the following $8$-dimensional cell complex 
$$ X = (S^3_a \vee S^3_b) \cup_{[a,[a,b]]} D^8 \cup_{[b,[a,b]]} D^8.$$
The group $\pi_{10}(K) \otimes \mathbb{Q}$ is rank 1 (see \cite[\S 13(d),(e)]{fht} for details); let $\tau$ be a generator. The claim is that $\tau$ can be detected by an iterated integral of length two, implying that $\tau$ has distortion $O(L^{11})$.

Label the two $3$-cells of $X$ by $a$ and $b$, and the two $8$-cells $w$ and $z$ respectively. The cells give a basis of homology and so dually we can find (for each cell $e \in \{a,b,w,z\}$) $\text{dim}(e)$-dimensional forms $\omega_e$ on $K$ giving a basis of $H^*(K; \mathbb{Q})$. Then we will show that $\smallint \omega_a \omega_z$ detects $\tau$.

We show this by considering the minimal model of $K$:
$$\mathcal{M}_K = \langle x_1^{(3)},\;x_2^{(3)},\;y^{(5)},\; T^{(10)},\;\dots \;|\; dx_i = 0,\;dy = x_1x_2,\;dT = x_1x_2y,\; \dots \rangle.$$
and the quasi-isomorphism $m_K: \mathcal{M}_K \to \Lambda^*(K)$ can be taken such that $x_1 \mapsto \omega_a$, $x_2 \mapsto \omega_b$, $y$ is sent to a primitive $\omega_y$ of $\omega_a \wedge \omega_b$ and all other generators are sent to zero. Note also that $x_1 y \mapsto \omega_w$ and $x_2 y \mapsto \omega_z$.

The minimal model for $\Omega K$ can also be deduced from the minimal model for $K$: 
$$\mathcal{M}_{\Omega K} = \langle \tilde{x}_1^{(2)},\;\tilde{x}_2^{(2)},\;\tilde{y}^{(4)},\; \tilde{T}^{(9)},\;\dots \;|\; d\tilde{x}_i = 0,\;d\tilde{y} = 0,\;d\tilde{T} = 0,\; \dots \rangle. $$
Note that $H^9(\Omega K; \mathbb{Q})$ is also rank 1, so any nontrivial class in here will detect $\tau$.

Now, let $A$ be the image of the minimal model $\mathcal{M}_K$ in the de Rham complex $\Lambda^*(K)$ of $K$. In degree $3$ it is generated by $\omega_a$ and $\omega_b$. In degree $5$ it is generated by $\omega_y$. In degree $8$ it is generated by $\omega_w = \omega_a \wedge \omega_y$ and $\omega_z = \omega_b \wedge \omega_y$. In other positive degrees $A$ vanishes. By Theorem 3.2, the subalgebra $\smallint A$ of $\Lambda^*(\Omega K)$ has cohomology isomorphic to $H^*(\Omega K)$. So we are looking for cocycles in $\smallint A$ of degree $9$. By inspection the degree $9$ part of $\smallint A$ has dimension $8$, with basis
\begin{align*} 
    \smallint \omega_a \omega_w, \hspace{1mm} \smallint \omega_a \omega_z, \hspace{1mm} \smallint \omega_b \omega_w, \hspace{1mm} \smallint \omega_b \omega_z, \hspace{1mm} \smallint \omega_y (\omega_a \wedge \omega_b), \\ 
    \smallint \omega_a  (\omega_a \wedge \omega_b) \omega_b, \hspace{1mm}  \smallint \omega_a  \omega_a (\omega_a \wedge \omega_b), \hspace{1mm}
    \smallint (\omega_a \wedge \omega_b) \omega_b \omega_b.
\end{align*}
Using the fact that $\omega_a^2 = 0 = \omega_b^2$ and also that $\omega_a \wedge \omega_b \wedge \omega_y = 0$ for degree reasons, we can take the exterior derivatives of these in $\smallint A$ and see the space of $9$-cocycles in $\smallint A$ is dimension $7$. The space of $9$-coboundaries is dimension $6$, with a basis given the exterior derivatives of:
\begin{align*}
    \smallint \omega_a \omega_y \omega_a - \smallint \omega_a \omega_a \omega_a \omega_b, \\
    \smallint \omega_b \omega_y \omega_b - \smallint \omega_a \omega_b \omega_b \omega_b, \\
    \smallint \omega_a \omega_y \omega_b - \smallint \omega_a \omega_a \omega_b \omega_b, \\
    \smallint \omega_a \omega_a \omega_a \omega_b, \\
    \smallint \omega_a \omega_a \omega_b \omega_b, \\
    \smallint \omega_a \omega_b \omega_b \omega_b 
\end{align*}

Hence $\smallint \omega_a \omega_z$ is a representative for the generator of $H^9(\Omega K)$. Since this iterated integral is length $2$, we get the upper bound on the distortion of $\tau$ of $O(L^{11})$ as required.
\end{example}

\section{Sharpness of bounds}

We show that the bounds on differential forms representing the degree and Hopf invariant for spheres are sharp.

\begin{proposition}
Let $\omega$ be any $(n-1)$-form on $\Omega S^n$ generating $H^{n-1}(\Omega S^n)$. Then it cannot be the case that $||\omega(\gamma)||_{\infty} = o(\text{Length}(\gamma))$ for $\gamma \in \Omega S^n$.
\end{proposition}

\begin{proof}
Suppose for contradiction we have such an $\omega$ with $||\omega(\gamma)||_{\infty} = o(\text{Length}(\gamma))$. Let $\xi$ be a cycle generating $H_{n-1}(\Omega S^n)$. Consider the map $\{L\}: \Omega S^n \to \Omega S^n$ sending $\gamma$ to $\gamma \cdot \ldots \cdot \gamma$; $\gamma$ to $\gamma$ concatenated with itself itself $L$ times. This is $1$-Lipschitz so the induced map on chains $\{L\}_*: C_*(\Omega S^n) \to C_*(\Omega S^n)$ does not increase volume. It does however increase the suplength of chains: for any chain $Z$, $\text{Suplength}(\{L\}_*Z) = L\cdot \text{Suplength}(Z)$. 

Note that $[\{L\}_*\xi] = L[\xi] \in H_{n-1}(\Omega S^n)$ since $\{L\}_*\xi$ sweeps out $S^n$ by loops $L$ times. 

Then $$L = \langle \omega, \{L\}_*\xi \rangle \leq \text{Vol}(\{L\}_*\xi) \cdot \sup_{\gamma \in \text{supp}(\{L\}_*\xi)} ||\omega(\gamma)||_{\infty} = o(L),$$ which is a contradiction.
\end{proof}

\begin{proposition}
Let $n$ be even and $\omega$ be any $(2n-2)$-form on $\Omega S^n$ generating $H^{2n-2}(\Omega S^n)$. Then it cannot be the case that $||\omega(\gamma)||_{\infty} = o(\text{Length}(\gamma)^2)$ for $\gamma \in \Omega S^n$.
\end{proposition}

\begin{proof} 
Let $\cdot$ denote the Pontryagin product on chains in $\Omega X$. Then $\xi \cdot \xi$ is a $(2n-2)$-cycle generating $H_{2n-2}(\Omega S^n)$. We proceed as before, this time using the family of cycles $\{L\}_*\xi \cdot \{L\}_*\xi$. By linearity of the Pontryagin product, $[\{L\}_*\xi \cdot \{L\}_*\xi] = L^2[\xi \cdot \xi] \in H_{2n-2}(\Omega S^n)$.

Suppose we have such an $\omega$ with $||\omega(\gamma)||_{\infty} = o(\text{Length}(\gamma)^2)$.
Then $$L^2 \lesssim \langle \omega, \{L\}_*\xi \cdot  \{L\}_*\xi \rangle \leq \text{Vol}(\{L\}_*\xi \cdot  \{L\}_*\xi) \cdot \sup_{\gamma \in \text{supp}(\{L\}_*\xi \cdot  \{L\}_*\xi)} ||\omega(\gamma)||_{\infty} = o(L^2),$$ a contradiction.
\end{proof}

The crucial properties of the families of cycles in the preceding two propositions were that they were $\text{suplength} \lesssim L$ and the ratio of their degree in homology to their volume was large (linear in $L$ and quadratic in $L$ respectively).

We can study this phenomenon in more generality. 

\begin{definition}
Let $X$ be a compact simply connected Riemannian manifold and $\alpha \in H_n(\Omega X; \mathbb{Q})$.  The \textit{homological distortion} of $\alpha$ is
\begin{align*}
\delta_{\alpha}'(L) \coloneqq \text{max}\{k \; : \; & \text{there exists a cycle representing $k\alpha$ that is} \\ & \text{suplength $\leq L$ and volume $\leq L^n$}\}.
\end{align*}
\end{definition}

The bound on the volume can always be satisfied by rescaling the coefficients of a cycle. The volume bound of $L^n$ is so that homological distortion can be compared to distortion:

\begin{proposition} For any $\alpha \in \pi_{n+1}(X) \otimes \mathbb{Q}$ with Hurewicz image $\tau(\alpha) \in H_{n}(\Omega X; \mathbb{Q})$, $\delta_{\alpha}(L) \lesssim \delta_{\tau(\alpha)}'(L)$. 
\end{proposition}

\begin{proof}
An $L$-Lipschitz map $f: S^{n+1} \to X$ with $[f] = k\alpha$, gives a cycle $\tau(f) \coloneqq (\Omega f)_*(\zeta_{n+1}) \in C_n(\Omega X)$ with volume $\lesssim L^n$, suplength $\lesssim L$ and $[\tau(f)] = k \tau(\alpha)$.
\end{proof}

Note that Corollary 4.3 really gives an upper bound on homological distortion which by Proposition 6.3 implies the upper bound on distortion. It is an open question as to whether there is an $\alpha$ such that $\delta_{\tau(\alpha)}'(L)$ grows strictly faster than $\delta_{\alpha}(L)$.

\printbibliography

\medskip

\textsc{Department of Mathematics, Massachusetts Institute of Technology, Cambridge, MA, United States} \\
E-mail address: \texttt{relliott@mit.edu}

\end{document}